\documentclass[a4paper,11pt]{article}
\usepackage{latexsym,amsfonts,amsmath}
\usepackage{amssymb,amsthm}
\usepackage{color,soul} 
\usepackage{xcolor}
\usepackage[most]{tcolorbox}  

\usepackage[mathscr]{eucal}
\usepackage{enumerate}
\usepackage{enumitem}
\usepackage{xparse}    
\usepackage{etoolbox}  
\usepackage[normalem]{ulem} 

\usepackage[hidelinks]{hyperref} 

\usepackage{todonotes} 

\usepackage{pgf}
\usepackage{tikz}
\usetikzlibrary{arrows,automata}

\usepackage{fullpage}

\usepackage{natbib}

\numberwithin{equation}{section} 

\definecolor{Dgreen}{rgb}{102,255,102}
\definecolor{Purp}{rgb}{102,0,204}

\def\1{\mathbf{1}}
\def\0{\mathbf{0}}
\def\inter{\mathop{\cap}}
\def\NN{\mathbb{N}}

\def\RR{\mathbb{R}}

\def\XX{\mathbf{X}}
\def\AA{\mathbf{A}}
\def\KK{\mathbf{K}}
\def\SS{\mathbf{S}}
\def\HH{\mathbf{H}}
\def\MM{\boldsymbol{\mathcal{M}}}

\newcommand{\widebar}[1]{\overline{#1}}

\newcommand{\bcal}[1]{\boldsymbol{\mathcal{#1}}}
\newcommand{\bscr}[1]{\boldsymbol{\mathscr{#1}}}
\newcommand{\bfrak}[1]{\boldsymbol{\mathfrak{#1}}}

\def\union{\mathop{\cup}}
\def\inter{\mathop{\cap}}

\def\ds{\displaystyle}

\DeclareMathOperator\aff{aff}
\DeclareMathOperator\conv{conv}
\DeclareMathOperator\ext{ext}

\DeclareMathOperator\rai{rai}

\def\card{\operatorname{card}}

\definecolor{Dgreen}{rgb}{0,0.5,0}

\newtheorem{theorem}{Theorem}[section]
\newtheorem{proposition}[theorem]{Proposition}
\newtheorem{lemma}[theorem]{Lemma}
\newtheorem{corollary}[theorem]{Corollary}
\newtheorem{definition}[theorem]{Definition}
\newtheorem{remark}[theorem]{Remark}

\theoremstyle{definition}

\renewcommand{\theassumption}{\Alph{assumption}} 
\newlist{assumptions}{enumerate}{1}
\setlist[assumptions,1]{label=(\theassumption.\arabic*), ref=(\theassumption.\arabic*), leftmargin=*}

\NewDocumentEnvironment{List-perso}{m}
  {
    \begin{enumerate}[label=(#1.\arabic*), ref=(#1.\arabic*), leftmargin=*]
  }
  {
    \end{enumerate}
  }

\NewDocumentEnvironment{myenv}{mm}
  {
    \par\medskip
    \noindent\textbf{#1}\par 

    \newlist{vlist}{enumerate}{1}
    \setlist[vlist,1]{label=(#2.\arabic*), ref=(#2.\arabic*), leftmargin=*}
  }
  {
    \par\medskip
  }
  
\newtcolorbox{textecouleur}[1][]{%
  colback=white,
  colframe=white,
  coltext=#1,
  width=\textwidth,
  left=0mm,
  right=0mm,
  box align=left,     
  boxrule=0pt,        
  breakable,
  enhanced,
  sharp corners,
  before skip=5pt,
  after skip=5pt
}
\begin{document}
\bibliographystyle{plain}

\title{Absorbing Markov Decision Processes: Geometric Properties and Sufficiency of Finite Mixtures of Deterministic Policies
\footnote{Supported by grant PID2021-122442NB-I00 from the Spanish \textit{Ministerio de Ciencia e Innovaci\'on.}}}

\author{
Fran\c{c}ois Dufour\footnote{
Bordeaux INP; Inria centre at the University of Bordeaux, Team: ASTRAL; IMB, Institut de Math\'ematiques de Bordeaux, Universit\'e de Bordeaux, France. e-mail:  \tt{francois.dufour@math.u-bordeaux.fr}  {(Author for correspondence)}}
\and
Tom\'as Prieto-Rumeau\footnote{Statistics Department, UNED, Madrid, Spain. e-mail: {\tt{tprieto@ccia.uned.es}}}}
\date{ \today }
\maketitle
\begin{abstract}
In this paper we investigate several geometric properties of the set of occupancy measures. 
In particular, we analyse the structure of the faces generated by a given occupancy measure,  together with their relative algebraic interior. We also determine the affine hulls of these 
faces and describe the associated parallel linear subspaces. It is shown that these structures can be fully characterised in terms of the parameters that define the underlying  Markov decision process (MDP).
Moreover, we establish that the class of finite mixtures of deterministic stationary policies constitutes a sufficient class of policies for uniformly 
absorbing MDPs with a measurable state space and multiple criteria.
We also provide a characterisation of the minimal order required for a finite mixture of deterministic stationary policies to represent the performance vector of an arbitrary policy.
\end{abstract}
{\small 
}

\section{Introduction}
\label{sec-1}
 
In this work, we study a discrete‑time absorbing Markov Decision Process (MDP) $\mathsf{M}$ defined over a measurable state space $\XX$ and a Borel action space $\AA$, equipped with a given initial distribution.
The set of admissible state–action pairs, denoted by $\KK$, satisfies certain technical measurability conditions.
The collection of all admissible policies is denoted by $\mathbf{\Pi}$.
We also focus on several specific subclasses: deterministic stationary policies, denoted by $\mathbb{D}$; chattering stationary policies, denoted by $\mathbb{C}$; and finite mixtures of deterministic stationary policies,
denoted by $\mathbb{F}$.
A stationary chattering policy is a stationary randomized policy that, for each state, randomizes among a finite fixed set of stationary deterministic policies according to a state‑dependent distribution (see Definition \ref{Def-policies} for a precise statement).
We have chosen to use the term of chattering policy which is borrowed from optimal control theory. Readers are referred to the discussion at the top of page 25 in \cite{balder95}
and the references therein for a discussion of this topic, as well as  \cite{feinberg2020} that used it in the context of MDPs.
A finite mixture of deterministic stationary policies corresponds to a policy whose occupancy measure can be written as a convex combination of a fixed finite set of occupancy measures generated by deterministic stationary policies
(see Definition \ref{Def-mixture-policies} for a precise statement). It is, in fact, a subset of chattering stationary policies, that is, $\mathbb{F}\subset \mathbb{C}$.

For a measurable subset $\Delta \subset \XX$, called the absorbing set, a $\Delta$‑absorbing MDP is one in which the state process remains in $\Delta$ indefinitely once entered, producing no additional reward or cost. A common assumption is that the expected hitting time to $\Delta$ is finite under every admissible policy. The notion of a uniformly absorbing MDP, introduced in \cite[Definition~3.6]{piunovskiy19}, intuitively requires that the tail probabilities of the hitting‑time series converge to zero uniformly over all policies. For comprehensive treatments and recent developments on (uniformly) absorbing MDPs, we refer to \cite{altman99,fra-tom2024,fra-tom2025,fra-tom-arxiv2025,piunovskiy19,feinberg12,Piunovskiy25-book,piunovskiy24,piunovskiy24b,yi24} and the references therein.

The occupancy measure $\mu_{\pi}$ induced by a policy $\pi\in\mathbf{\Pi}$ captures the expected cumulative behavior of the state‑action process over time. More precisely, for any measurable set $\Gamma \subset \XX\times\AA$, $\mu_{\pi}(\Gamma)$ represents the expected total time the process spends in $\Gamma$ under policy $\pi$. We denote by $\bcal{O}$ the set of all occupancy measures.

A bounded one‑step reward vector‑function $r:\XX\times\AA\to\RR^{d}$, with $d\in\NN^{\*}$, is used to assess policy performance. The performance vector of a policy $\pi\in\mathbf{\Pi}$ is then given by
$$
\mathcal{R}(\pi)=\int_{\XX\times\AA} r(x,a)\,\mu_{\pi}(dx,da).
$$
A subset $\varLambda\subset\mathbf{\Pi}$ is said to form a sufficient class of policies if it satisfies $\mathcal{R}(\varLambda)=\mathcal{R}(\mathbf{\Pi})$.

\bigskip

Our first objective is to study some geometric properties of the set of occupancy measures. In particular, we analyse the structure of the faces generated by $\mu\in\bcal{O}$, as well as their relative algebraic interior. 
We characterise the affine hulls of these faces, together with their associated parallel linear subspaces.
The relative algebraic interior of a face generated by a point of convex set has been studied in detail in \cite{weis25,weis21}. In particular, the author has in \cite{weis25} obtained several useful result for faces and their relative algebraic interior and has applied them to spaces of probability mesures (see sections 9, 10 and 11 in \cite{weis25}). One may also mention \cite[page~239]{dubins62} for an analysis of the faces generated by a probability measure defined on a general measurable space.
Our results partially build on the work in \cite{weis25}, but also on our recent results in \cite[Sections~3 and~4]{fra-tom2024} where the set of occupancy measures for an absorbing MDP has been completely described in terms of the characteristic equation and an additional condition stating that an occupancy measure must be absolutely continuous with respect to a reference probability measure.
One of the key contributions is to show that a face, its affine hull and its associated parallel subspace can be fully characterised in terms of the parameters defining the Markov‑controlled model $\mathsf{M}$. More precisely, they can be described in terms of $Q\mathbb{I}_{\Delta^{c}}$‑invariant signed measures $\nu$ defined on $\XX\times\AA$ that are equivalent to $\mu$ and for which the Radon–Nikodym derivatives $\frac{d\nu}{d\mu}$ and $\frac{d\mu}{d\nu}$ are bounded (see Proposition \ref{Elementary-face} and Corollary \ref{aff-rai-elementary-face}).

\bigskip

The second aim of this work is to investigate, within the class of uniformly absorbing models, whether the family of finite mixtures of deterministic stationary policies is sufficient. Mixtures of policies have been studied in \cite{feinberg96} in a general context. A policy $\pi$ is called a mixture of policies from $\varGamma\subset\mathbf{\Pi}$ if there exists a probability measure $\nu$ on the set of strategic measures generated by $\varGamma$ such that the strategic measure induced by $\pi$ can be expressed as the barycentre of $\nu$. In \cite{feinberg96}, it was shown that any (respectively Markov) policy is a mixture of deterministic (respectively Markov) policies. The mixture is called finite if, in the preceding definition, the measure $\nu$ can be written as a convex combination of Dirac measures. Finite mixtures of deterministic stationary policies have been investigated in \cite{feinberg12} and also in \cite{piunovskiy97}. In \cite{feinberg12}, it was shown for constrained uniformly absorbing MDPs with Borel state space and compact action sets that a stationary optimal policy exists and can be taken as a finite mixture of deterministic stationary policies. A procedure to compute such an optimal policy as a mixture of finitely many deterministic stationary policies was also proposed. Related results can be found in \cite{altman96} for constrained MDPs with countable state space. In \cite{piunovskiy97}, the family of finite mixtures of deterministic stationary policies was studied in detail; the questions of denseness and optimality were addressed under suitable continuity–compactness conditions.
A substantial body of work is devoted to the study of mixtures of deterministic stationary policies in MDPs. For a thorough and detailed overview of this literature, we refer the reader to the following references \cite{feinberg96,feinberg12,piunovskiy97}.

We will show that for a uniformly absorbing model $\mathsf{M}$, one can construct, for any policy $\pi\in\mathbf{\Pi}$, a finite mixture $\gamma$ of deterministic stationary policies of order $d+1$ that attains the same performance vector, namely,
$$
\int_{\XX\times\AA} r(x,a)\,\mu_{\pi}(dx,da)=\int_{\XX\times\AA} r(x,a)\,\mu_{\gamma}(dx,da),
$$
or, equivalently, $\mathcal{R}(\mathbb{F}_{d+1})=\mathcal{R}(\mathbf{\Pi})$, as stated in Theorem~\ref{Main-theorem}. To the best of our knowledge, this assertion is novel even when the reward function is real‑valued. We also provide a characterisation of the minimal order required for a finite mixture of deterministic stationary policies to represent the performance vector of an arbitrary policy $\pi\in\mathbf{\Pi}$;
see Proposition~\ref{Optimal-order}.

\bigskip

The rest of the paper is organized as follows. In the remaining of this section we introduce some notation.
In Section \ref{Sec-2}, we introduce the control model under consideration and provide some basic definitions as well as basic properties.
Section \ref{Sec-3} is devoted to the analysis of the geometric properties of the set of occupancy measures and in particular to the characterization of its faces.
Under a set of appropriate assumption, we will establish in section \ref{Sec-4} the sufficiency of the finite mixtures of deterministic stationary policies.

\paragraph{Notation and terminology}
$\NN$ is the set of natural numbers including $0$, $\NN^{*}=\NN-\{0\}$, $\RR$ denotes the set of real numbers, $\RR_{+}$ the set of non-negative real numbers,
$\RR_{+}^{*}=\RR_{+}-\{0\}$, $\widebar{\RR}_{+}=\RR_{+}\union \{+\infty\}$ and $\widebar{\RR}_{+}^*=\RR_{+}^*\union \{+\infty\}$.
For any $q\in \NN$, $\NN_{q}$ is the set $\{0,1,\ldots,q\}$ and for any $q\in \NN^{*}$, $\NN_{q}^{*}$ is the set $\{1,\ldots,q\}$.
We write $\bscr{S}_{p}=\big\{(\beta_{1},\ldots,\beta_{p})\in\RR_{+}^p : \sum_{i=1}^p \beta_{i}=1 \big\}$ for the standard simplex in $\RR^p$
and $0_p$ for the zero vector in $\mathbb{R}^p$ for $p \in \mathbb{N}^*$. If $\{u_{n}\}_{n\in\NN}$ is a sequence of real numbers, we will use the following convention 
$\sum_{n=0}^{-1} u_{n}=0$ to simplify the writing of certain formulae.

On a measurable space $(\mathbf{\Omega},\bcal{F})$ we will consider the set of finite signed measures $\bcal{M}(\mathbf{\Omega})$, the set $\bcal{M}^+(\mathbf{\Omega})$
of finite nonnegative measures, 
and the set of probability measures~$\bcal{P}(\mathbf{\Omega})$.
An element of $\bcal{M}^+(\mathbf{\Omega})$ will be simply called a measure.
Let $(\mathbf{\Omega},\bcal{F},\mu)$ be a measure space. A set $\Gamma\subset\mathbf{\Omega}$ is called a $\mu$-null set if there exists $\Lambda\in\bcal{F}$ such that
$\Gamma\subset\Lambda$ and $\mu(\Lambda)=0$.
For a set $\Gamma\in\bcal{F}$, we denote by $\mathbf{I}_{\Gamma}:\mathbf{\Omega}\rightarrow\{0,1\}$ the indicator function of the set~$\Gamma$, that is,
$\mathbf{I}_{\Gamma}(\omega)=1$ if and only if $\omega\in\Gamma$.
For $\omega\in\mathbf{\Omega}$, we write $\delta_{\{\omega\}}$ for the Dirac probability measure at $\omega$ defined on $(\mathbf{\Omega},\bcal{F})$ by
$\delta_{\{\omega\}}(B)=\mathbf{I}_{B}(\omega)$ for any $B\in\bcal{F}$.

Let $(\mathbf{\Omega},\bcal{F})$ and $(\widetilde{\mathbf{\Omega}},\widetilde{\bcal{F}})$ be two measurable spaces.
A kernel on $\widetilde{\mathbf{\Omega}}$ given $\mathbf{\Omega}$ is a mapping
\mbox{$Q:\mathbf{\Omega}\times\widetilde{\bcal{F}}\rightarrow\RR^+$} such that $\omega\mapsto Q(B|\omega)$ is 
measurable on $(\mathbf{\Omega},\bcal{F})$ for every set $B$ in $\widetilde{\bcal{F}}$,   and  $B\mapsto Q(B|\omega)$ is in $\bcal{M}^+(\widetilde{\mathbf{\Omega}})$
for every $\omega\in\mathbf{\Omega}$. Depending on the context, and in particular when it is important to specify the $\sigma$-algebra associated with $\mathbf{\Omega}$, we will say that $Q$ is a kernel 
on $\widetilde{\mathbf{\Omega}}$ given $(\mathbf{\Omega},\bcal{F})$. Moreover, if $Q$ is a kernel on $\mathbf{\Omega}$ given $\mathbf{\Omega}$, we will simply say that
$Q$ is a kernel on $\mathbf{\Omega}$.
If $Q(\widetilde{\mathbf{\Omega}}|\omega)=1$ for all $\omega\in\mathbf{\Omega}$ then we say that $Q$ is a \textit{stochastic} kernel.
We now introduce two types of specific kernels we will work with.
\begin{enumerate}
\item A kernel $\gamma$ on $\widetilde{\mathbf{\Omega}}$ given $\mathbf{\Omega}$ is called \textit{finitely supported}  if there exist an integer $p\in\NN^{*}$ and $p$ measurable functions
$\{\phi_{i}\}_{i\in\NN_{p}^{*}}$ defined from $\mathbf{\Omega}$ to $\widetilde{\mathbf{\Omega}}$ and a measurable function $\beta$ defined 
from $\mathbf{\Omega}$ to $\bcal{S}_{p}$ satisfying
$$ \gamma(d\widetilde{\omega} | \omega) =\sum_{i=1}^{p} \beta_{i}(\omega) \delta_{\phi_{i}(\omega)}(d\widetilde{\omega})$$
for any $\omega\in\mathbf{\Omega}$ where $\beta(\omega)=(\beta_{1}(\omega),\ldots,\beta_{p}(\omega))$.
If it is needed to specify the number $p$ in the decomposition of $\gamma$, we will say that $\gamma$ is \textit{finitely supported of order $p$}.
This terminology is borrowed from the theory of Young measures, see Definition 8.1 in \cite{balder95}.
\item For $\Gamma\in\bcal{F}$, we write $\mathbb{I}_{\Gamma}$ for the kernel on $\mathbf{\Omega}$ defined by
$\mathbb{I}_{\Gamma}(B|\omega)=\mathbf{I}_{\Gamma}(\omega) \delta_{\{\omega\}}(B)$ for $\omega\in\mathbf{\Omega}$ and $B\in\bcal{F}$.
\end{enumerate}
Let $Q$ be a stochastic kernel on $\widetilde{\mathbf{\Omega}}$ given $\mathbf{\Omega}$.
For  a bounded measurable function $f:\widetilde{\mathbf{\Omega}}\rightarrow\RR^p$ or a measurable function $f:\widetilde{\mathbf{\Omega}}\rightarrow\widebar{\RR}^p_{+}$
with $p\in\NN^*$, we will denote by $Qf$ the measurable function defined on $\mathbf{\Omega}$ by
$Qf=(Qf_{1},\ldots,Qf_{i},\ldots,Qf_{p})$ where $\ds Qf_{i}(\omega)=\int_\mathbf{\widetilde{\Omega}} f_{i}(\widetilde{\omega})Q(d\widetilde{\omega}|\omega)$ where $f_{i}$ is the $i$-th component of $f$
for $\omega\in\mathbf{\Omega}$, $i\in\{1,\cdots,p\}$.
In the same spirit, for a measure $\mu\in\bcal{M}^{+}(\mathbf{\Omega})$, we write $\mu(f)$ for the vector $\big(\mu(f_{1}),\ldots,\mu(f_{i}),\ldots,\mu(f_{p})\big)$
where $\ds \mu(f_{i})=\int_{\mathbf{\Omega}} f_{i}(\omega) \mu(d\omega)$.
We also denote by $\mu Q$ the finite measure  on $(\widetilde{\mathbf{\Omega}},\widetilde{\bcal{F}})$ given by
$\ds B\mapsto \mu Q\,(B)= \int_{\mathbf{\Omega}} Q(B|\omega) \mu(d\omega)$ for $B\in\widetilde{\bcal{F}}$.
Let $(\widebar{\mathbf{\Omega}},\widebar{\bcal{F}})$ be a third measurable space and $R$ be a stochastic kernel on $\widebar{\mathbf{\Omega}}$
given $\widetilde{\mathbf{\Omega}}$. Then we will denote by $QR$ the stochastic kernel on $\widebar{\mathbf{\Omega}}$ given $\mathbf{\Omega}$ defined as
$\ds QR(\Gamma |\omega)= \int_{\widetilde{\mathbf{\Omega}}} R(\Gamma | \tilde{\omega}) Q(d\tilde{\omega} | \omega)$
for $\Gamma\in\widebar{\bcal{F}}$ and $\omega\in\bcal{F}$.
The product of the $\sigma$-algebras $\bcal{F}$ and $\widetilde{\bcal{F}}$ is denoted by $\bcal{F}\otimes\widetilde{\bcal{F}}$ and consists of the $\sigma$-algebra
generated by the measurable rectangles, that is, the sets of the form $\Gamma\times\widetilde{\Gamma}$ for $\Gamma\in\bcal{F}$ and
$\widetilde{\Gamma}\in\widetilde{\bcal{F}}$.
We denote by $\mu\otimes Q$ the unique  finite measure on the product space $(\mathbf{\Omega}\times\widetilde{\mathbf{\Omega}},\bcal{F}\otimes\widetilde{\bcal{F}})$ satisfying 
$\ds \mu\otimes Q( \Gamma\times\widetilde{\Gamma})= \int_{\Gamma} Q(\widetilde{\Gamma}|\omega)\mu(d\omega)$ for $\Gamma\in\bcal{F}$ and
$\widetilde{\Gamma}\in\widetilde{\bcal{F}}$.
Given $\mu\in\bcal{M}(\mathbf{\Omega}\times\widetilde{\mathbf{\Omega}})$, 
$\mu^{\mathbf{\Omega}}(\cdot)=\mu(\cdot\times \widetilde{\mathbf{\Omega}})\in\bcal{M}(\mathbf{\Omega})$ and
$\mu^{\widetilde{\mathbf{\Omega}}}(\cdot)=\mu(\mathbf{\Omega}\times\cdot)\in\bcal{M}(\widetilde{\mathbf{\Omega}})$ are the marginal measures.

For a metric space $\SS$, we write $\bfrak{B}(\SS)$ for its Borel $\sigma$-algebra. A metric space will be always endowed with its Borel $\sigma$-algebra.
A subset of a metric space will always be equipped with the induced metric unless explicitly stated otherwise.
For a measurable space $(\mathbf{\Omega},\bcal{F})$ and a metric space $\SS$, $\bcal{L}^{0}_{\SS}(\mathbf{\Omega},\bcal{F})$ stands for the set of measurable functions from $(\mathbf{\Omega},\bcal{F})$
into $(\SS,\bfrak{B}(\SS))$. If there is no ambiguity about the $\sigma$-algebra associated with $\mathbf{\Omega}$, we will write $\bcal{L}^{0}_{\SS}(\mathbf{\Omega})$ to simplify notation.
If $\mathbf{S}$ is a Polish space (a complete and separable metric space), we will consider on $\bcal{M}(\mathbf{\Omega}\times\mathbf{S})$ the  $ws$-topology (weak-strong topology)
which is the coarsest topology for which the mappings $\mu\mapsto \mu(f)$ are continuous for any bounded Carath\'eodory function $f$ defined on $\mathbf{\Omega}\times\mathbf{S}$.

\bigskip

There are several definitions of extreme set in the literature. We use here Definition 7.61 in \cite{aliprantis06}. Observe that this definition is different from that used in \cite{dubins62} and \cite{klee63}.
An extreme subset of a subset $\mathcal{C}$ of a vector space $\mathcal{V}$ is a nonempty subset $\mathcal{E}$ of $\mathcal{C}$ with the property that
if $x$ belongs to $\mathcal{E}$ it cannot be written as a proper convex combination of points of $\mathcal{C}$ outside $\mathcal{E}$.
A point $x$ is an extreme point of $\mathcal{C}$ if the singleton $\{x\}$ is an extreme set.
The set of extreme points of a subset $\mathcal{C}$ of a vector space $\mathcal{V}$ is denoted by $\ext(\mathcal{C})$.
A face of $\mathcal{C}$ is an extreme subset of $\mathcal{C}$ that is itself convex.
A convex set $\mathcal{C}$ in a vector space $\mathcal{V}$ is linearly bounded (linearly closed, respectively) if the intersection of $\mathcal{C}$ with any line in $\mathcal{V}$ is a bounded (closed, respectively) set.
The relative (respectively, boundary) interior of a convex set $\mathcal{C}$ in a topological vector space $\mathcal{V}$ is defined to be its topological (respectively, boundary) interior relative to its affine hull.
We write $\mathrm{aff}\{K\}$ for the affine hull of a convex set $K$ and $\rai[K]$ for the relative algebraic interior of $K$, that is, the set of points $x$ in $K$ such that for every line $g$ in 
$\aff(K)$ containing $x$, the intersection $g\cap K$ includes an open segment containing $x$.

\section{The Markov control model}
\label{Sec-2}
The main goal of this section is to introduce the parameters defining the control model with a brief presentation of the construction of the controlled process.
We will present different types of properties for a control model, such as being uniformly absorbing. 
We will also introduce the notion of occupancy measure, which has been studied in detail in 
\cite{fra-tom2024,fra-tom2025,fra-tom-arxiv2025,piunovskiy24} for absorbing models and
which will play a fundamental role in the proof techniques we will use in this paper.

\subsection{The model.}
\label{Description-model}
A control model  $\mathsf{M}$ consists of the following parameters.
\begin{itemize}
\item A state space $\XX$ endowed with a $\sigma$-algebra $\bfrak{X}$. The set $\XX$ will always be endowed with the $\sigma$-algebra $\bfrak{X}$ unless explicitly stated, in which case the context will be specified.
\item A Borel space $\mathbf{A}$, representing the action space.
\item A family of nonempty measurable sets $\AA(x)\subseteq \AA$ for $x\in\XX$. The set $\AA(x)$ gives the admissible actions in state~$x$.
Let $\KK=\{(x,a)\in\XX\times\AA: a\in \AA(x)\}$ be the family of feasible state-action pairs.
\item A stochastic kernel $Q$ on $\mathbf{X}$ given $\big(\XX\times\AA,\bfrak{X}\otimes\bfrak{B}(\AA)\big)$ which stands for the transition probability function. 
\item An initial distribution given by $\eta\in\bcal{P}(\XX)$.
\item A bounded reward vector-function $r\in\bcal{L}^{0}_{\RR^{d}}(\XX\times\AA)$ for $d\in\NN^*$ that will be used to evaluate the performance of a control policy.
\end{itemize}
We often write a model $\mathsf{M}$ as $\mathsf{M}=(\mathbf{X},\mathbf{A},\{\mathbf{A}(x)\}_{x \in \mathbf{X}},Q,\eta,r)$ to specify the notations used for the parameters of this model.

\begin{definition}
We will say that a model $\mathsf{M}=(\mathbf{X},\mathbf{A},\{\mathbf{A}(x)\}_{x\in \mathbf{X}},Q,\eta,r)$ satisfies the \textit{measurability conditions} if the following properties holds.
\begin{List-perso}{M}
\item \label{K-measurability} $\KK$ is a measurable subset of $\big(\XX\times\AA,\bfrak{X}\otimes\bfrak{B}(\AA)\big)$.
\item \label{selector-theta} There exists a function $\theta\in\bcal{L}^{0}_{\AA}(\XX,\bfrak{X})$ whose graph is a subset of $\KK$.
\end{List-perso}
\end{definition}

\bigskip

Let us denote by $\mathbb{D}$ the set of functions $\phi\in \bcal{L}^{0}_{\AA}(\XX,\bfrak{X})$ satisfying $\phi(x)\in\AA(x)$ for any $x\in \XX$ and
by $\mathbb{S}$ (respectively, $\mathbb{C}$) the set of (respectively, finitely supported) stochastic kernels $\varphi$ on $\AA$ given $\XX$ satisfying $\varphi(\AA(x)|x)=1$
for any $x\in \XX$.
We will write $\mathbb{C}_{p}$ for the set of finitely supported kernels in $\mathbb{C}$ of order $p\in\NN^{*}$.

\bigskip

Let us consider now an arbitrary but fixed model $\mathsf{M}=(\mathbf{X},\mathbf{A},\{\mathbf{A}(x)\}_{x\in \mathbf{X}},Q,\eta,r)$ satisfying the measurability conditions.
The space of admissible histories of the controlled process up to time $n\in\NN$ is denoted by $\HH_{n}$. It is defined recursively by 
$\HH_0=\XX\quad\hbox{and}\quad
 \HH_n=\HH_{n-1}\times\AA\times\XX$ for $n\ge1$, all endowed with their corresponding product $\sigma$-algebras.
 A control policy $\pi$ is a sequence $\{\pi_n\}_{n\ge0}$ of stochastic kernels on $\AA$ given $\HH_n$, denoted by $\pi_n(da|h_n)$, such that
$\pi_n(\AA(x_n)|h_n)=1$ for each $n\ge0$ and $h_n=(x_0,a_0,\ldots,x_n)\in\HH_n$. 
The set of all policies is denoted by $\mathbf{\Pi}$. 
\begin{definition}
\label{Def-policies}
We now present several specific families of policies that will be used in the remainder of this work.
\begin{itemize}
\item
A policy $\pi=\{\pi_{n}\}_{n\in\NN}$ is called a \emph{randomized stationary policy} if there exists $\varphi\in\mathbb{S}$ satisfying
$\pi_{n}(\cdot |h_{n})=\varphi(\cdot |x_{n})$ for any $h_{n}=(x_{0},a_{0},\ldots,x_{n})\in \mathbf{H}_{n}$ and $n\in \NN$.
By a slight abuse of notation, we will write $\mathbb{S}$ for the set of all randomized stationary policies.
\item
A policy $\pi=\{\pi_{n}\}_{n\in\NN}$ is called a \emph{chattering stationary policy} if
there exists a finitely supported kernel $\varphi\in\mathbb{C}$ satisfying
$\pi_{n}(\cdot |h_{n})=\varphi(\cdot |x_{n})$ for $h_{n}=(x_{0},a_{0},\ldots,x_{n})\in \mathbf{H}_{n}$ and $n\in \NN$.
A chattering stationary policy defined by $\varphi\in\mathbb{C}$ will be said of order $p\in\NN^{*}$ if $\varphi$ is of order $p$.
By a slight abuse of notation, we will write $\mathbb{C}_{p}$ for the set of all chattering stationary policies of order $p$ for some $p\in\NN^{*}$.
\item
A policy $\pi=\{\pi_{n}\}_{n\in\NN}\in \mathbf{\Pi}$ is called a \emph{deterministic stationary policy} if there exists $\phi\in\mathbb{D}$ satisfying
$\pi_{n}(\cdot |h_{n})=\delta_{\phi(x_{n})}(\cdot)$ for any $h_{n}=(x_{0},a_{0},\ldots,x_{n})\in \mathbf{H}_{n}$ and $n\in \NN$.
We will identify the set of deterministic stationary policies with the set $\mathbb{D}$. We will therefore use $\mathbb{D}$ to designate these two sets.
\end{itemize}
\end{definition}

The canonical space of all possible sample paths of the state-action process is 
$\mathbf{\Omega}=(\XX\times\AA)^{\infty}$ endowed with the product $\sigma$-algebra $\bcal{F}$. 
The coordinate projection functions from $\mathbf{\Omega}$ to the state space $\XX$, the action space $\AA$, and $\HH_n$ for $n\ge0$ are respectively denoted   by $X_{n}$, $A_{n}$, and $H_n$. 
We will refer  to $\{X_{n}\}_{n\in \NN}$ as the state process and $\{A_{n}\}_{n\in \NN}$ as the action  process.
It is a well known result that for every policy $\pi \in \mathbf{\Pi}$ 
there exists a unique probability
measure $\mathbb{P}_{\pi}$ on $(\mathbf{\Omega},\bcal{F})$
such that $ \mathbb{P}_{\pi} (\KK^{\infty})=1$ and such that for every $n\in\NN$, $\Gamma\in \bfrak{X}$, and  $ \Lambda\in \bfrak{B}(\mathbf{A})$ we have 
$\mathbb{P}_{\pi}(X_{0}\in \Gamma)=\eta(\Gamma)$, 
$$\mathbb{P}_{\pi}(X_{n+1}\in \Gamma\mid H_{n},A_{n}) = Q(\Gamma\mid X_{n},A_{n}) \quad\hbox{and}\quad
\mathbb{P}_{\pi}(A_{n}\in \Lambda\mid H_{n})= \pi_{n}(\Lambda\mid H_{n})$$
with
$\mathbb{P}_{\pi}$-probability one.
We will refer to $\mathbb{P}_\pi$ as to a \emph{strategic probability measure}, and we will denote by $\bcal{S}=\{\mathbb{P}_{\pi}\}_{\pi\in\mathbf{\Pi}}\subseteq\bcal{P}(\mathbf{\Omega})$
the set of all strategic probability measures. 
The expectation with respect to  $\mathbb{P}_{\pi}$ is denoted by $\mathbb{E}_{\pi}$. For a $p$-dimensional random variable $Z=(Z_{1},\ldots,Z_{i},\ldots,Z_{p})$ with $p\in\NN^{*}$,
we will write $\mathbb{E}_{\pi}[Z]$ for the vector $\big(\mathbb{E}_{\pi}[Z_{1}],\ldots,\mathbb{E}_{\pi}[Z_{i}],\ldots,\mathbb{E}_{\pi}[Z_{p}]\big)$ when $\mathbb{E}_{\pi}[Z_{i}]$ is well defined for any $i\in\{1,\ldots,p\}$.

\subsection{Occupancy measures and absorbing models}
In this section, we consider an arbitrary but fixed model $\mathsf{M}=(\mathbf{X},\mathbf{A},\{\mathbf{A}(x)\}_{x\in \mathbf{X}},Q,\eta,r)$ satisfying the measurability conditions.
We are going to introduce  absorbing  models. To do this, we will need to define the notion of occupancy measures associated with a control policy.

\begin{definition}
\label{absorbing}
We say that the control model  $\mathsf{M}$ is $\Delta$-absorbing if $\Delta\in\bfrak{X}$ and the following conditions are satisfied
\begin{enumerate}
\item[$\bullet$] $Q(\Delta|x,a)=1$ and $r(x,a)=0_{d}$ for $(x,a)\in\Delta\times\AA$.
\item[$\bullet$] $\mathbb{E}_{\pi}[T_\Delta]$ is finite for any $\pi\in\mathbf{\Pi}$.
\end{enumerate}
and we say that  it is uniformly $\Delta$-absorbing if, additionally, the following condition holds
\begin{enumerate}
\item[$\bullet$] $\ds \lim_{n\rightarrow\infty} \sup_{\pi\in\mathbf{\Pi}} \sum_{t=n}^\infty\mathbb{P}_{\pi}\{T_\Delta>t\}=0$.
\end{enumerate}
where $T_\Delta:\mathbf{\Omega}\rightarrow\NN\cup\{\infty\}$ is defined by
$T_\Delta(\omega)=\min\{n\ge0: X_{n}(\omega)\in\Delta\},$ where by convention the $\min$ over the empty set is defined as $+\infty$.
\end{definition}
The notion of a uniformly $\Delta$-absorbing MDP was introduced in \cite[Definition 3.6]{piunovskiy19}. 

\bigskip

We define now the occupancy measures of a $\Delta$-absorbing control model $\mathsf{M}$. 
\begin{definition}\label{def-occupancy-measure}
For a $\Delta$-absorbing control model $\mathsf{M}$, the  occupancy measure $\mu_{\pi}\in\MM^{+}(\XX\times\AA)$ of a control policy $\pi\in\mathbf{\Pi}$ is 
$$\mu_{\pi}(\Gamma) = \mathbb{E}_{\pi}\Big[ \sum_{t=0}^\infty \mathbf{I}_{\{T_\Delta>t\}}\cdot\mathbf{I}_{\{(X_t,A_{t})\in \Gamma\}}\Big] \quad\hbox{for $\Gamma\in\bfrak{X}\otimes\mathfrak{B}(\AA)$.}$$
\end{definition}
\begin{remark}
The measure $\mu_{\pi}$ is a finite measure. Indeed, \mbox{$\mu_{\pi}(\XX\times\AA)= \mathbb{E}_{\pi}[T_{\Delta}]<+\infty$}
since $\mathsf{M}$ is by assumption $\Delta$-absorbing.
\end{remark}
We will write $\bcal{O}=\{\mu_{\pi}:\pi\in\mathbf{\Pi}\}$ for the set of occupancy measures. If $\mathbf{\Lambda}\subseteq\mathbf{\Pi}$ is some subset of policies,  we will use the notation
$\bcal{O}_{\mathbf{\Lambda}}=\{\mu_{\pi}:\pi\in\mathbf{\Lambda}\}$.

\begin{remark}
\label{div-remark}
\begin{enumerate}[label=(\alph*)]
\item Observe that the subset $\bcal{O}$ of the vector space $\MM(\XX\times\AA)$ is convex.
This well known property has been established in the special case of a Borel state space see for example Corollary 4.3 in \cite{feinberg12} and Proposition 3 in \cite{piunovskiy24}.
For the general case of a measurable state space, we can proceed by using similar arguments in showing that the set of strategic measures $ \bcal{S}$ is convex  implying the convexity of $\bcal{O}$.
\item \label{div-remark-b} It is also important to recall the following well-known property: $\bcal{O}=\bcal{O}_{\mathbb{S}},$
see for example Theorem 8.1 in \cite{altman99} for a countable state space, Lemma 4.1 in \cite{feinberg12} or Lemma 2 in \cite{piunovskiy24} for a Borel state space
and Proposition 3.3(ii) in \cite{fra-tom2024} for a general measurable state space.
\end{enumerate}
\end{remark}

\bigskip

\begin{definition}
\label{Def-mixture-policies}
We will denote by $\mathbb{F}$ the set of stationary control policies $\gamma\in\mathbb{S}$ for which there exists some
 nonnegative real numbers $\alpha_1,\ldots,\alpha_n\ge0$ with $\sum\alpha_i=1$ and some deterministic stationary policies $\gamma_1,\ldots,\gamma_n\in\mathbb{D}$ such that 
\begin{align}
\label{sec2-Def-policy-F}
\mu_{\gamma}= \sum_{j=1}^{n} \alpha_{j} \mu_{\gamma_{j}}.
\end{align}
We refer to $\mathbb{F}$ as the set of \emph{finite mixtures of deterministic stationary policies}. Equivalently,
$$\mathbb{F}=\conv\big( \bcal{O}_{\mathbb{D}} \big).$$
For a given integer $n\ge1$, let $\mathbb{F}_n$ be the set of policies that are finite mixtures of $n$ deterministic stationary as in \eqref{sec2-Def-policy-F}.
\end{definition}
The advantage of such a policy is that it can be very easily generated.
It consists in adding an initial randomization step (choosing randomly an integer $\tau$ in $\NN_{n}^*$ according to the distribution 
$\{\alpha_{j}\}_{j\in\NN_{n}^{*}}$) and then select the policy (within the finite family $\{\gamma_{j}\}_{j\in\NN_{n}^{*}}$ of deterministic stationary policies in $\mathbb{D}$)  given by $\gamma_{\tau}\in \mathbb{D}$.

\bigskip

Now, we introduce the so-called characteristic equation. 
\begin{definition} Suppose that the model $\mathsf{M}$ is $\Delta$-absorbing.
A measure $\mu$ in $\MM^{+}(\XX\times\AA)$ is a solution of the characteristic equations when 
$\mu(\KK^c)=0$ and $\mu^\XX=(\eta+\mu Q)\mathbb{I}_{\Delta^c}$.
We denote by $\bcal{C}$ the set of all solutions of the characteristic equations.
\end{definition}

We need to extend the definition of $Q\mathbb{I}_{\Delta^c}$-invariant measure that was introduced in \cite{fra-tom2024} for nonnegative measures $\vartheta\in\bcal{M}^+(\XX\times\AA)$ to the set of finite signed measures.
\begin{definition}
\label{Def-invariant}
A measure $\vartheta\in\bcal{M}(\XX\times\AA)$ is said $Q\mathbb{I}_{\Delta^c}$-invariant when
$$\vartheta^\XX= \vartheta Q\mathbb{I}_{\Delta^c}.$$
\end{definition}

The performance of a policy $\pi\in\mathbf{\Pi}$ is evaluated by using the following vector-function.
\begin{definition}
Suppose that the model $\mathsf{M}$ is $\Delta$-absorbing. The vector of expected rewards for the policy $\pi\in\mathbf{\Pi}$ denoted by $\mathcal{R}(\pi)$ is given by
$$\mathcal{R}(\pi)=\mathbb{E}_{\pi} \Big[ \sum_{t=0}^\infty r(X_t,A_t)\Big].$$
This defines an application from $\mathbf{\Pi}$ to $\RR^{d}$ called the  performance vector function.
A subset $\mathbf{\Lambda}\subset\mathbf{\Pi}$ will be called a sufficient class (family) of policies if $\mathcal{R}(\mathbf{\Lambda})=\mathcal{R}(\mathbf{\Pi})$.
\end{definition}

\begin{remark} As in \cite[Remark 2.14]{fra-tom-arxiv2025},  
there is no loss of generality to assume that the action space $\mathbf{A}$ is compact and so we will make this assumption in the forthcoming. \end{remark}

Throughout the rest of this paper, we will use a fixed model $\mathsf{M}=(\mathbf{X},\mathbf{A},\{\mathbf{A}(x)\}_{x\in \mathbf{X}},Q,\eta,r)$ satisfying the measurability conditions.

\section{Geometric properties of the set of occupancy measures}
\label{Sec-3}
\def\codim{\operatorname{codim}}
\def\im{\operatorname{im}}

In this section, we derive several geometric properties of the set of occupancy measures. 
In particular, we examine:
\begin{itemize}
  \item the faces $\bscr{F}(\mu)$ of $\bcal{O}$ generated by a given $\mu\in\bcal{O}$,
  \item for $\alpha\in\RR^{d}$, the faces $\bscr{F}_{\alpha}(\mu)$ of $\bcal{O}(r,\alpha)$ generated by a given $\mu\in\bcal{O}(r,\alpha)$,
\end{itemize}
where, for $\alpha\in\RR^{d}$,
$$
\bcal{O}(r,\alpha)=\bcal{O}\cap\bcal{H}(r,\alpha),\qquad
\bcal{H}(r,\alpha)=\bigl\{\nu\in\MM(\XX\times\AA)\colon \nu(r)=\alpha\bigr\}.
$$

We also describe the relative algebraic interiors of these faces, we characterize their affine hulls, and we identify the associated parallel linear subspaces.

A central result of this section is that the faces $\bscr{F}(\mu)$ and $\bscr{F}_{\alpha}(\mu)$, their affine hulls
and their associated parallel subspaces
can be described completely in terms of the parameters that define the controlled Markov model $\mathsf{M}$.
For instance, we show that the linear subspace parallel to the affine hull of the face $\bscr{F}(\mu)$ generated by a given $\mu\in\bcal{O}$ consists precisely of the $Q\mathbb{I}_{\Delta^{c}}$‑invariant signed measures
$\nu$ (see Definition~\ref{Def-invariant}) that are equivalent to $\mu$, and whose Radon–Nikodym derivatives $\dfrac{d\nu}{d\mu}$ and $\dfrac{d\mu}{d\nu}$ are bounded.

\bigskip

We will need the following technical result.
\begin{lemma}
\label{Equivalence-occup-measure-2}
Suppose that the Markov controlled model $\mathsf{M}$ is $\Delta$-absorbing.
Consider $\mu, \nu \in\bcal{O}$ and $\varepsilon>0$. Then,
$$\mu + \varepsilon (\mu - \nu) \in \bcal{O} \iff \mu + \varepsilon (\mu - \nu) \in \bcal{M}^{+}(\XX\times\AA).$$
\end{lemma}
\begin{proof}
The direct implication is obvious. Let us show the converse implication. 
Assume that $\mu + \varepsilon (\mu - \nu) \in \bcal{M}^{+}(\XX\times\AA)$. Then $\mu + \varepsilon (\mu - \nu) \in \bcal{C}$
and $\mu + \varepsilon (\mu - \nu) \ll \dfrac{1}{2}(\mu + \nu)$. 
According to Remark \ref{div-remark}(a), $\dfrac{1}{2}(\mu + \nu)\in \bcal{O}$ and so, Lemma 2.9 in \cite{fra-tom-arxiv2025} yields $\mu + \varepsilon (\mu - \nu)\in\bcal{O}$ giving the result.
\end{proof}

\begin{proposition}
\label{Elementary-face}
Suppose that the Markov controlled model $\mathsf{M}$ is $\Delta$-absorbing.
For $\mu\in \bcal{O}$, the face of $\bcal{O}$ generated by $\mu$ is given by
$$\bscr{F}(\mu)=\Big\{\nu\in\bcal{C} : \nu\leq c \mu \text{ for some } c >1\Big\}.$$
The relative algebraic interior of $\bscr{F}(\mu)$ is
$$\mathrm{rai}[\bscr{F}(\mu)]=\Big\{\nu\in\bcal{C} : \frac{1}{c} \mu \leq \nu\leq c \mu \text{ for some } c >1\Big\}.$$
The affine hull of $\bscr{F}(\mu)$ is defined by
\begin{align}
\label{Affine-F}
\gamma\in \mathrm{aff}\{\bscr{F}(\mu)\}  \iff
\begin{cases}
\gamma\in\bcal{M}(\XX\times\AA), \\
(\gamma-\mu)^{\XX} = (\gamma-\mu) Q\mathbb{I}_{\Delta^c}, \\
-c \mu \leq \gamma-\mu \leq c \mu, \text{ for some } c\in\RR_{+}^{*}.
\end{cases}
\end{align}
The linear subspace of $\MM(\XX\times\AA)$ parallel to $\mathrm{aff}\{\bscr{F}(\mu)\}$ is given by
\begin{align}
\label{subspace-Affine-F}
\mathscr{V}(\mu)=\big\{\nu\in\MM(\XX\times\AA) : \nu^{\XX} = \nu Q \mathbb{I}_{\Delta^c} \text{ and }-c \mu \leq \nu\leq c \mu, \text{ for some } c\in\RR_{+}^{*} \big\}
\end{align}
\end{proposition}
\begin{proof}
From Proposition 4.1 in \cite{weis25}, $\nu\in\bscr{F}(\mu)$ if and only if $\nu\in\bcal{O}$ and $\mu+\varepsilon (\mu-\nu) \in\bcal{O}$ for some
$\varepsilon>0$. 
For $\varepsilon>0$, Lemma  \ref{Equivalence-occup-measure-2} yields
$$\nu\in\bscr{F}(\mu) \iff \nu\in\bcal{O} \text{ and } \nu\leq  \frac{(1+\varepsilon)}{\varepsilon} \mu \iff \nu\in\bcal{C} \text{ and } \nu\leq  \frac{(1+\varepsilon)}{\varepsilon} \mu$$
where the last equivalence holds by virtue of Lemma 2.9 in \cite{fra-tom-arxiv2025}. This gives the first item.

For $\nu\in\bscr{F}(\mu)$, Corollary 5.7 in \cite{weis25} yields $\nu\in \mathrm{rai}[\bscr{F}(\mu)]$ if and only if $\nu\in \bscr{F}(\mu) \text{ and } \mu\in \bscr{F}(\nu)$
and combined with the previous item this gives the second claim.

By Corollary 4.6 in \cite{weis25}, $\gamma\in\mathrm{aff}[\bscr{F}(\mu)]$ if and only if 
$\gamma\in\MM(\XX\times\AA)$ and 
$\mu\pm \varepsilon (\gamma-\mu) \in \bcal{O}$, for some $\varepsilon \in\RR_{+}^{*}$.
Observe that for $\gamma\in\MM(\XX\times\AA)$ and $\varepsilon \in\RR_{+}^{*}$, we have
$\mu\pm \varepsilon (\gamma-\mu) \in \bcal{O}$ if and only if $-\dfrac{1}{\varepsilon} \mu \leq \gamma-\mu \leq \dfrac{1}{\varepsilon} \mu$
and $(\gamma-\mu)^{\XX}=(\gamma-\mu)Q\mathbb{I}_{\Delta^{c}}$, this last equivalence comes from Lemma 2.9 in \cite{fra-tom-arxiv2025}. This shows the claim.

For the proof of the last item, observe that $\mathscr{V}(\mu)$ is clearly a linear subspace of $\MM(\XX\times\AA)$.
Now the result is an immediate consequence of the expression of $\mathrm{aff}\{\bscr{F}(\mu)\}$ given in \eqref{Affine-F}.
Indeed, $$\gamma\in \mathrm{aff}\{\bscr{F}(\mu)\} \iff (\gamma-\mu) \in\mathscr{V}(\mu)$$ for $\mu\in\bcal{O}$ giving the result.
\end{proof}

Here is a corollary of the preceding results that characterizes the extremal structures associated with a face generated by a measure in the constrained‑measure set $\bcal{O}(r,\alpha)$ for $\alpha\in\RR^d$.
\begin{corollary}
\label{aff-rai-elementary-face}
Suppose that the Markov controlled model $\mathsf{M}$ is $\Delta$-absorbing.
Let us consider $\alpha\in\RR$ such that $\bcal{O}(r,\alpha)\neq\emptyset$.
Let us write $\bscr{F}_{\alpha}(\mu)$ the face of the convex set $\bcal{O}(r,\alpha)$ generated by $\mu\in \bcal{O}(r,\alpha)$.
Then, we have
\begin{align*}
\bscr{F}_{\alpha}(\mu)=\Big\{\nu\in\bcal{C} : \nu(r)=\alpha \text{ and } \nu\leq c \mu \text{ for some } c >1\Big\}
\end{align*}
and 
\begin{align*}
\mathrm{rai}[\bscr{F}_{\alpha}(\mu)]=\Big\{\nu\in\bcal{C} : \nu(r)=\alpha \text{ and } \frac{1}{c} \mu \leq \nu\leq c \mu \text{ for some } c >1\Big\}
\end{align*}
and
\begin{align*}
\gamma\in \mathrm{aff}\{\bscr{F}_{\alpha}(\mu)\}  \iff
\begin{cases}
\gamma\in\bcal{M}(\XX\times\AA), \\
\gamma(r)=\alpha, \\
(\gamma-\mu)^{\XX} = (\gamma-\mu) Q\mathbb{I}_{\Delta^c}, \\
-c \mu \leq \gamma-\mu \leq c \mu, \text{ for some } c\in\RR_{+}^{*}.
\end{cases}
\end{align*}
and the linear subspace of $\MM(\XX\times\AA)$ parallel to $\mathrm{aff}\{\bscr{F}_{\alpha}(\mu)\}$ is given by
\begin{align*}
\mathscr{V}_{\alpha}(\mu)=\big\{\nu\in\MM(\XX\times\AA) : \nu(r)=0, \: \nu^{\XX} = \nu Q \mathbb{I}_{\Delta^c} \text{ and }-c \mu \leq \nu\leq c \mu, \text{ for some } c\in\RR_{+}^{*} \big\}.
\end{align*}
\end{corollary}
\begin{proof}
Observe that the face of $\bcal{H}(r,\alpha)$ generated by $\mu\in \bcal{H}(r,\alpha)$, its relative algebraic interior and its affine hull are given by $\bcal{H}(r,\alpha)$.
Applying Proposition 5.8 in \cite{weis25}, the results are an easy consequence of Proposition \ref{Elementary-face}.
\end{proof}

\section{Sufficiency of finite mixtures of deterministic stationary policies and minimal order}
\label{Sec-4}
In this section we first establish the sufficiency of the class of finite mixtures of deterministic stationary policies, i.e.  
$$
\mathcal{R}(\mathbb{F})=\mathcal{R}(\mathbf{\Pi}),
$$ 
for a uniformly absorbing model (see Theorem~\ref{D-Sigma-sufficiency}).  Using different arguments, this result is sharpened in Theorem~\ref{Main-theorem} by showing that the subclass $\mathbb{F}_{d+1}$ forms a sufficient family for such models.  Finally, we characterize the minimal order required for a finite mixture of deterministic stationary policies to reproduce the performance vector of an arbitrary policy
$\pi\in\mathbf{\Pi}$ (see Proposition~\ref{Optimal-order}).

\bigskip

We first need to establish the following technical result.
\begin{lemma}
\label{lissage}
Let $\Phi\in \mathbb{C}_{p}$ for $p\in\NN^*$ then there exists $\Phi^{*}\in \mathbb{C}_{p}$ supported on $\{\phi^{*}_{i}(x)\}_{i\in\NN^*_{p}}$ for some
$\{\phi^{*}_{i}\}_{i\in\NN^*_{p}}\subset \mathbb{D}$ satisfying
\begin{enumerate}[label=\arabic*)]
\item $\Phi(\cdot | x)=\Phi^{*}(\cdot | x)$ for any $x\in\XX$,
\item $\Phi^{*}(\{\phi^{*}_{i}(x)\} | x)>0$ for any $x\in\XX$ and $i\in\NN_{p}^{*}$.
\end{enumerate}
\end{lemma}
\begin{proof}
There exist $p$ measurable functions $\{\phi_{i}\}_{i\in\NN_{p}^{*}}$ in $\mathbb{D}$ and a measurable function $\beta$ defined from $\XX$ to $\bcal{S}_{p}$ satisfying
$\ds \Phi(\cdot | x) =\sum_{i=1}^{p} \beta_{i}(x) \delta_{\phi_{i}(x)}(\cdot)$
for any $x\in\XX$ where $\beta(x)=(\beta_{1}(x),\ldots,\beta_{p}(x))$.
Let us define the multifunction $\mathcal{Z}$ from $\XX$ into $\NN_{p}^{*}$ by $\mathcal{Z}(x)=\{k\in\NN_{p}^{*} : \beta_{k}(x) = 0\}$ which is measurable by Corollary 18.8 in \cite{aliprantis06}.
Combining Lemma 18.7 and Theorem 18.19 in \cite{aliprantis06}, it follows that the function $\tau$ from $\XX$ into $\NN_{p}^{*}$ defined by $\tau(x)=\min \{j\in\NN_{p}^{*} : \beta_{k}(j) > 0\}$ is measurable.
The kernel $\Phi^{*}(\cdot | x)$ defined by $\ds \sum_{i=1}^{p} \beta^{*}_{i}(x) \delta_{\phi^{*}_{i}(x)}(\cdot)$
where 
\begin{align*}
\beta^{*}_{i}(x) =
\begin{cases}
\frac{\beta_{\tau(x)}(x)}{\card(\mathcal{Z}(x))+1} & \text{ if } i\in \mathcal{Z}(x)\union\{\tau(x)\}, \\
\beta_{i}(x) & \text{ otherwise},
\end{cases}
\quad \text{and} \quad
\phi^{*}_{i}(x) =
\begin{cases}
\phi_{\tau(x)}(x) & \text{ if } i\in \mathcal{Z}(x), \\
\phi_{i}(x) & \text{ if } i\in \NN_{p}^{*}\setminus\mathcal{Z}(x),
\end{cases}
\end{align*}
satisfies the required properties.
\end{proof}

\bigskip

The following result establishes the sufficiency of $\mathbb{F}$.
\begin{theorem}
\label{D-Sigma-sufficiency}
Assume that the model $\mathsf{M}$ is uniformly $\Delta$-absorbing. Then, $\mathbb{F}$ is a sufficient family of policies, that is, $\mathcal{R}(\mathbb{F})=\mathcal{R}(\mathbf{\Pi})$.
\end{theorem}
\begin{proof}
From Theorem 5.6 in \cite{fra-tom-arxiv2025}, it is sufficient to establish $\mathcal{R}(\mathbb{F})=\mathcal{R}(\mathbf{C}_{d+1})$ to get the result.
Therefore, let us consider an arbitrary policy $\Phi\in \mathbb{C}_{p+1}$. Lemma \ref{lissage} shows the existence of $\Phi^{*}\in \mathbb{C}_{p+1}$ supported on
$\{\phi^{*}_{i}(x)\}_{i\in\NN^*_{p+1}}$ for some $\{\phi^{*}_{i}\}_{i\in\NN^*_{p+1}}\subset \mathbb{D}$ satisfying $\Phi^{*}(\{\phi^{*}_{i}(x)\} | x)>0$ for any $x\in\XX$ and $i\in\NN_{p+1}^{*}$ and 
$\mu_{\Phi}=\mu_{\Phi^{*}}$.
From, the proof of Theorem 5.4 in \cite{fra-tom-arxiv2025}, it follows that the model $\mathsf{M}^{*}=(\mathbf{X},\mathbf{A},\{\mathbf{A}^{*}(x)\}_{x\in \mathbf{X}},Q,\eta,r)$
with $\AA^{*}(x)=\{\phi^{*}_{i}(x) : i\in\NN_{p}^{*} \}$ is precisely one of those that satisfies the conclusions of Theorem 5.4 in \cite{fra-tom-arxiv2025}.
From Theorem 5.5 in \cite{fra-tom-arxiv2025}, $\mu_{\Phi^{*}}(r)$ belongs to the relative interior of $\mathcal{R}(\mathbf{\Pi}^{*})$ where
$\mathbf{\Pi}^{*}\subset\mathbf{\Pi}$ denotes set of all control policies associated with the model $\mathsf{M}^{*}$.
Writing  $\bcal{O}^{*}$ for the set of occupancy measures of $\mathsf{M}^{*}$ we have that
$\bcal{O}^{*}$ is $ws$-compact as well as $\mathcal{R}(\mathbf{\Pi}^{*})$.
The Krein-Milman Theorem yields that 
$$\bcal{O}^{*}=\widebar{\conv\big(\bcal{O}^{*}_{\mathbb{D}^{*}}\big)}.$$
where $\mathbb{D}^{*}$ is the set of deterministic stationary policies for the model $\mathsf{M}^{*}$.
Let $\mathscr{R}^{*}$ be the linear mapping defined on $\bcal{O}^{*}$ by $\mathscr{R}^{*}(\mu)=\mu(r)$. Then, it is continuous and so,
the convex set $\mathscr{R}^{*}\Big(\conv\big(\bcal{O}^{*}_{\mathbb{D}^{*}}\big)\Big)$ is dense in $\mathscr{R}^{*}(\bcal{O}^{*})$. Therefore, they have the same relative interior. This yields that
$$\mu_{\Phi}(r)=\mu_{\Phi^{*}}(r)\in\mathscr{R}^{*}\Big(\conv\big(\bcal{O}^{*}_{\mathbb{D}^{*}}\big)\Big).$$
However, $\bcal{O}^{*}_{\mathbb{D}^{*}}\subset\bcal{O}_{\mathbb{D}}$ and so, $\mu_{\Phi}(r)\in\mathscr{R}\Big(\conv\big(\bcal{O}_{\mathbb{D}}\big)\Big)$. Since
$\bcal{O}_{\mathbb{F}}=\conv\big(\bcal{O}_{\mathbb{D}}\big)$, this shows the claim.
\end{proof}

\bigskip

The next theorem establishes a link between the order of a finite mixture of deterministic stationary policies and the dimension of face in $\bcal{O}$ generated by its occupancy measure.
More precisely, it states that that for an occupancy measure $\mu\in \bcal{O}$ we have $\dim \big( \bscr{F}(\mu)\big)\leq p$ if and only if $\mu\in\bcal{O}_{\mathbb{F}_{p+1}}$.
\begin{theorem}
\label{finite-face-equivalence}
Suppose the model $\mathsf{M}$ is uniformly $\Delta$-absorbing.
\begin{enumerate}[label=\arabic*)]
\item If $\dim \big( \bscr{F}(\mu)\big)\leq p$ for some $\mu\in\bcal{O}$ and $p\in\NN^{*}$, then $\mu\in\bcal{O}_{\mathbb{F}_{p+1}}$.
\item If $\mu\in\bcal{O}_{\mathbb{F}_{p+1}}$ for some $p\in\NN^{*}$, then $\dim \big( \bscr{F}(\mu)\big)\leq p$.
\end{enumerate}
\end{theorem}
\begin{proof}
Let us show the first item. 
Observe that Proposition 3.2 in \cite{fra-tom-arxiv2025} yields $\bscr{F}(\mu)$ is linearly closed and linearly bounded since $\bscr{F}(\mu)$ is a face of $\bcal{O}$. 
Then $(3)$ of  \cite{klee63} shows that any $\gamma\in \bscr{F}(\mu)$ can be written as a convex combination of at most $p+1$ extreme points of $\bscr{F}(\mu)$ that are also extreme points of 
$\bcal{O}$ since $\bscr{F}(\mu)$ is a face of $\bcal{O}$. Therefore,  $\mu\in \bcal{O}_{\mathbb{F}_{p+1}}$ showing the result.

To prove the second item. Consider $\mu\in \bcal{O}_{\mathbb{F}_{p+1}}$. Then there exists $\{\gamma_{j}\}_{j\in\NN_{p+1}^{*}}\subset \mathbb{D}$ and \mbox{$(\alpha_{j})_{j\in\NN_{p+1}^{*}}\in\bscr{S}_{p+1}$}
such that $\ds \mu= \sum_{j=1}^{p+1} \alpha_{j} \mu_{\gamma_{j}}$. Therefore, $\bscr{F}(\mu)\subset\conv\big( \{\mu_{\gamma_{1}},\ldots,\mu_{\gamma_{p+1}}\}\big)$
implying $\dim \big( \bscr{F}(\mu)\big)\leq p$.
\end{proof}

Now the next theorem refines the result established in Theorem \ref{D-Sigma-sufficiency} by showing that the collection of finite mixtures of deterministic stationary policies of order at most $d+1$ is a sufficient family of policies. An interesting point is that our proof relies on arguments different from those employed in the proof of Theorem~\ref{D-Sigma-sufficiency}.
\begin{theorem}
\label{Main-theorem}
Consider $\alpha\in\mathcal{R}(\mathbf{\Pi})$ and suppose the model $\mathsf{M}$ is uniformly $\Delta$-absorbing.
Then $\mathbb{F}_{d+1}$ is a sufficient family of policies, that is, $\mathcal{R}(\mathbb{F}_{d+1})=\mathcal{R}(\mathbf{\Pi})$.
\end{theorem}
\begin{proof}
From Theorem 5.6 in \cite{fra-tom-arxiv2025}, there exists $\mu_{\gamma}\in\bcal{O}(r,\alpha)$ with
$\gamma\in\mathbb{C}_{d+1}$ such that $\gamma(\cdot|x)$ is supported on $\{\gamma_{i}(x)\}_{i\in\NN^*_{d+1}}$ for some $\{\gamma_{i}\}_{i\in\NN^*_{d+1}}\subset \mathbb{D}$.
Let us introduce
$$\widetilde{\KK}=\big\{(x,a)\in\XX\times\AA : a\in\{\gamma_{i}(x): i\in\NN^*_{d+1}\}\}\big\}$$
and write  
$$\widetilde{\bcal{O}}=\{\mu\in \bcal{O}: \mu(\widetilde{\KK}^{c})=0\}.$$
We have clearly, $\mu_{\gamma}\in\widetilde{\bcal{O}}\inter\bcal{H}(r,\alpha)$ and so, Lemma 5.1 in \cite{fra-tom-arxiv2025} yields the existence of $\mu^{\alpha}\in \ext \big( \bcal{O}(r,\alpha) \big)$
Therefore, $\mu_{\alpha}$ is a convex combination of at most $d+1$ extreme points of $\bcal{O}$ by Theorem 3.5 in \cite{fra-tom-arxiv2025}.
However, $\ext(\bcal{O})=\bcal{O}_{\mathbb{D}}$ implying that $\mu_{\alpha}\in \mathbb{F}_{d+1}$ showing the result.
\end{proof}

For any $\alpha \in \mathcal{R}(\Pi)$, the following result gives a characterization of the minimal order of a finite mixture of deterministic stationary policies $\gamma$
that satisfies $\mu_{\gamma}(r)=\alpha$.
\begin{proposition}
\label{Optimal-order}
Assume that the model $\mathsf{M}$ is uniformly $\Delta$-absorbing.
For $\alpha\in\mathcal{R}(\mathbf{\Pi})$, the minimal order of the finite mixtures of deterministic stationary policies $\gamma\in\mathbb{F}$ satisfying $\mu_{\gamma}(r)=\alpha$ is given by $p^*+1$ with
$$p^*= \min_{\mu\in \bcal{O}(r,\alpha)} \dim\big( \mathscr{V}(\mu) \big)=\min_{\mu\in \bcal{O}(r,\alpha)} \Big\{\dim\big(\ker (\mathscr{R}_{\mu})\big) + \dim\big(\im (\mathscr{R}_{\mu})\big) \Big\} \leq d$$
where $\mathscr{V}(\mu)$ is given in equation \eqref{subspace-Affine-F} and  $\mathscr{R}_{\mu}$ is the linear mapping defined on $\mathscr{V}(\mu)$ by $\mathscr{R}_{\mu}(\nu)=\nu(r)$.
\end{proposition}
\begin{proof}
Consider $\alpha\in\mathcal{R}(\mathbf{\Pi})$.
Let us define $p^*$ as the minimum of $\dim\big( \mathscr{V}(\mu) \big)$ for $\mu\in\bcal{O}(r,\alpha)$.
Combining the second item of Theorem \ref{finite-face-equivalence} and Theorem \ref{Main-theorem}, we obtain that there exists
$\mu_{\alpha}$ in $\bcal{O}(r,\alpha)$ such that $\dim \big( \bscr{F}(\mu_{\alpha})\big)\leq d$.
 We have $p^{*}\leq d$ since by Proposition \ref{Elementary-face}, we have $\dim \big( \bscr{F}(\mu)\big)=\dim\big( \mathscr{V}(\mu) \big)$ for $\mu\in \bcal{O}(r,\alpha)$.
Let $\mu^{*}\in \bcal{O}(r,\alpha)$ be such that
$\dim\big( \mathscr{V}(\mu^{*}) \big)=p^{*}$. Then, the first item of Theorem \ref{finite-face-equivalence} shows that $\mu^{*}\in\mathbb{F}_{p^{*}+1}$.
Now, if there exists $p<p^{*}$ such that $\mu\in\mathbb{F}_{p+1}$ for some $\mu\in \bcal{O}(r,\alpha)$, then $\dim \big( \bscr{F}(\mu)\big)=\dim\big( \mathscr{V}(\mu) \big)=p<p^{*}$ by the second item of 
Theorem \ref{finite-face-equivalence} yielding a contradiction.
Therefore, $p^{*}$ is the minimal order of a finite mixture of deterministic stationary policies $\gamma$ that satisfies $\mu_{\gamma}(r)=\alpha$.
The last equality is easily obtained by observing that for $\mu\in \bcal{O}(r,\alpha)$ with $\dim\big( \mathscr{V}(\mu) \big)\leq d$ we have 
$$\dim\big(\mathscr{V}(\mu)\big) = \dim\big(\mathscr{V}_{\alpha}(\mu)\big) + \codim_{\mathscr{V}(\mu)} \big(\mathscr{V}_{\alpha}(\mu)\big) = \dim\big(\ker (\mathscr{R}_{\mu})\big) + \dim\big(\im (\mathscr{R}_{\mu})\big)
$$
where $\mathscr{V}_{\alpha}(\mu)$ is the linear subspace of $\mathscr{V}(\mu)$ given in Corollary \ref{aff-rai-elementary-face} showing the result.
\end{proof}
\begin{remark}
This gives an interesting geometric interpretation of the minimal representation of a policy with a given performance
vector as a finite mixture of deterministic policies which is stated in terms of the MDP's parameter. This is related to the minimal dimension of the faces of the policies with such performance vector.
\end{remark}


\begin{thebibliography}{10}

\bibitem{aliprantis06}
C.D. Aliprantis and K.C. Border.
\newblock {\em Infinite dimensional analysis}.
\newblock Springer, Berlin, 2006.

\bibitem{altman96}
E.~Altman.
\newblock Constrained {M}arkov decision processes with total cost criteria:
  occupation measures and primal {LP}.
\newblock {\em Math. Methods Oper. Res.}, 43(1):45--72, 1996.

\bibitem{altman99}
E.~Altman.
\newblock {\em Constrained {M}arkov decision processes}.
\newblock Stochastic Modeling. Chapman \& Hall/CRC, Boca Raton, FL, 1999.

\bibitem{balder95}
E.J. Balder.
\newblock Lectures on {Y}oung measures.
\newblock Cahier math\'ematiques de la d\'ecision 9517, \mbox{CEREMADE},
  Universit\'e Paris Dauphine, 1995.
\newblock \newline Available at:
  \mbox{\url{https://webspace.science.uu.nl/~balde101/lyoungm.pdf}}.

\bibitem{dubins62}
L.E. Dubins.
\newblock On extreme points of convex sets.
\newblock {\em J. Math. Anal. Appl.}, 5:237--244, 1962.

\bibitem{fra-tom2024}
F.~Dufour and T.~Prieto-Rumeau.
\newblock Absorbing {M}arkov decision processes.
\newblock {\em ESAIM Control Optim. Calc. Var.}, 30:Paper No. 5, 18, 2024.

\bibitem{fra-tom2025}
F.~Dufour and T.~Prieto-Rumeau.
\newblock Absorbing {M}arkov decision processes and their occupation measures.
\newblock {\em SIAM J. Control Optim.}, 63(1):676--698, 2025.

\bibitem{fra-tom-arxiv2025}
F.~Dufour and T.~Prieto-Rumeau.
\newblock On the \mbox{Feinberg-Piunovskiy} theorem and its extension to
  chattering policies, 2025.
\newblock arXiv paper. Available at: \url{https://arxiv.org/abs/2510.04808}.

\bibitem{feinberg96}
E.~Feinberg.
\newblock On measurability and representation of strategic measures in {M}arkov
  decision processes.
\newblock In {\em Statistics, probability and game theory}, volume~30 of {\em
  IMS Lecture Notes Monogr. Ser.}, pages 29--43. Inst. Math. Statist., Hayward,
  CA, 1996.

\bibitem{feinberg2020}
E.A. Feinberg, A.~Ja\'{s}kiewicz, and A.~S. Nowak.
\newblock Constrained discounted {M}arkov decision processes with {B}orel state
  spaces.
\newblock {\em Automatica J. IFAC}, 111:108582, 11, 2020.

\bibitem{piunovskiy19}
E.A. Feinberg and A.~Piunovskiy.
\newblock Sufficiency of deterministic policies for atomless discounted and
  uniformly absorbing {MDP}s with multiple criteria.
\newblock {\em SIAM J. Control Optim.}, 57(1):163--191, 2019.

\bibitem{feinberg12}
E.A. Feinberg and U.G. Rothblum.
\newblock Splitting randomized stationary policies in total-reward {M}arkov
  decision processes.
\newblock {\em Math. Oper. Res.}, 37(1):129--153, 2012.

\bibitem{klee63}
V.~Klee.
\newblock On a theorem of {D}ubins.
\newblock {\em J. Math. Anal. Appl.}, 7:425--427, 1963.

\bibitem{Piunovskiy25-book}
A.~Piunovskiy.
\newblock {\em Counterexamples in {M}arkov decision processes}, volume~6 of
  {\em Series on Optimization and its Applications}.
\newblock World Scientific Publishing Co., Hackensack, NJ, 2025.

\bibitem{piunovskiy24}
A.~Piunovskiy and Y.~Zhang.
\newblock Extreme occupation measures in {M}arkov decision processes with an
  absorbing state.
\newblock {\em SIAM J. Control Optim.}, 62(1):65--90, 2024.

\bibitem{piunovskiy24b}
A.~Piunovskiy and Y.~Zhang.
\newblock On the continuity of the projection mapping from strategic measures
  to occupation measures in absorbing {M}arkov decision processes.
\newblock {\em Appl. Math. Optim.}, 89(3):Paper No. 58, 25, 2024.

\bibitem{piunovskiy97}
A.~B. Piunovskiy.
\newblock {\em Optimal control of random sequences in problems with
  constraints}, volume 410 of {\em Mathematics and its Applications}.
\newblock Kluwer Academic Publishers, Dordrecht, 1997.
\newblock With a preface by V. B. Kolmanovskii and A. N. Shiryaev.

\bibitem{weis25}
S.~Weis.
\newblock A note on faces of convex sets.
\newblock {\em J. Convex Anal.}, 32(3):901--919, 2025.

\bibitem{weis21}
S.~Weis and M.~Shirokov.
\newblock The face generated by a point, generalized affine constraints, and
  quantum theory.
\newblock {\em J. Convex Anal.}, 28(3):847--870, 2021.

\bibitem{yi24}
Y.~Zhang and X.~Zheng.
\newblock Further remarks on absorbing {M}arkov decision processes.
\newblock {\em Oper. Res. Lett.}, 57:Paper No. 107191, 7, 2024.

\end{thebibliography}

\end{document}